\documentclass[a4paper,10pt]{amsart}
\usepackage{amsmath,amssymb,amsthm}
\usepackage{mathrsfs}

\usepackage[color = blue!20, bordercolor = black, textsize = tiny]{todonotes}
\usepackage[colorlinks,linkcolor=black,citecolor=black,urlcolor=black]{hyperref}
\usepackage{tikz-cd}

\numberwithin{equation}{section}
\newtheorem{thm}{Theorem}[section]
\newtheorem{cor}[thm]{Corollary}
\newtheorem{lem}[thm]{Lemma}
\newtheorem{prop}[thm]{Proposition}
\theoremstyle{definition}
\newtheorem{df}[thm]{Definition}

\newtheorem{const}[thm]{Construction}

\newcommand{\A}{\mathbb{A}}

\newcommand{\N}{\mathbb{N}}
\renewcommand{\P}{\mathbb{P}}

\newcommand{\Z}{\mathbb{Z}}

\newcommand{\cA}{\mathcal{A}}

\newcommand{\cC}{\mathcal{C}}
\newcommand{\cD}{\mathcal{D}}

\newcommand{\cF}{\mathcal{F}}
\newcommand{\cG}{\mathcal{G}}

\newcommand{\cO}{\mathcal{O}}

\newcommand{\rD}{\mathrm{D}}

\newcommand{\rL}{\mathrm{L}}

\newcommand{\rN}{\mathrm{N}}

\newcommand{\sT}{\mathscr{T}}

\DeclareMathOperator{\Hom}{Hom}
\DeclareMathOperator{\Spec}{Spec}

\newcommand{\id}{\mathrm{id}}
\newcommand{\ul}{\underline}
\newcommand{\ol}{\overline}

\newcommand{\Sm}{\mathrm{Sm}}
\newcommand{\lSm}{\mathrm{lSm}}
\newcommand{\SmlSm}{\mathrm{SmlSm}}
\newcommand{\lSch}{\mathrm{lSch}}

\newcommand{\Adm}{\mathrm{Adm}}
\newcommand{\Bl}{\mathrm{Bl}}
\newcommand{\Gr}{\mathrm{Gr}}

\newcommand{\Sp}{\mathrm{Sp}}
\newcommand{\PSh}{\mathrm{PSh}}
\newcommand{\Sh}{\mathrm{Sh}}
\newcommand{\Fun}{\mathrm{Fun}}
\newcommand{\cofib}{\mathrm{cofib}}
\newcommand{\tcofib}{\mathrm{tcofib}}

\newcommand{\colim}{\mathop{\mathrm{colim}}}

\newcommand{\logKGL}{\mathbf{logKGL}}
\newcommand{\KGL}{\mathbf{KGL}}

\newcommand{\eff}{\mathrm{eff}}
\newcommand{\SH}{\mathrm{SH}}
\newcommand{\logSH}{\mathrm{logSH}}
\newcommand{\unit}{\mathbf{1}}
\newcommand{\logDM}{\mathrm{logDM}}
\newcommand{\DM}{\mathrm{DM}}

\title{On the log motivic stable homotopy groups}

\author{Doosung Park}
\address{Department of Mathematics and Informatics, University of Wuppertal, Germany}
\email{dpark@uni-wuppertal.de}

\begin{document}
\begin{abstract}
We compare the log motivic stable homotopy category and the usual motivic stable homotopy category over a perfect field admitting resolution of singularities.
As a consequence,
we show that the log motivic stable homotopy groups are isomorphic to the usual motivic stable homotopy groups.
\end{abstract}

\maketitle

\section{Introduction}

For a perfect field $k$,
author's joint work \cite{logDM} with Binda-{\O}stv{\ae}r provides a stable $\infty$-category of log effective motives $\logDM^{\eff}(k)$,
which is a place for incorporating various non $\A^1$-invariant cohomology like Hodge cohomology into the motivic framework.
For every separated noetherian log smooth fs log scheme $X\in \lSm/k$,
the log motive $M(X)\in \logDM^{\eff}(k)$ can be associated.
Using this,
one can also define log motivic cohomology of $X$.

A natural question in such a non $\A^1$-invariant theory of motives is whether $\logDM^{\eff}(k)$ contains Voevodsky's stable $\infty$-category of motives $\DM^\eff(k)$ \cite{MVW} as a full subcategory or not.
An affirmative answer implies that the theory of log motives is an extension of the usual theory of motives.
This is answered in \cite[Theorem 8.2.16]{logDM} under the assumption of resolution of singularities on $k$.
With this assumption,
a consequence is that log motivic cohomology of $Y\in \lSm/k$ is isomorphic to the usual motivic cohomology of $Y-\partial Y\in \Sm/k$,
where $Y-\partial Y$ denotes the largest open subscheme of $Y$ with the trivial log structure.
The proof of this result relies on Voevodsky's theory of homotopy invariant sheaves with transfers \cite[\S 21-24]{MVW}.

In author's joint work \cite{logSH} with Binda-{\O}stv{\ae}r,
the log motivic stable $\infty$-category $\logSH(k)$ is introduced.
Again,
a natural question is whether $\logSH(k)$ contains the stable motivic $\infty$-category $\SH(k)$ of Morel-Voevodsky \cite{MV} as a full subcategory or not.
We show the following result using a different approach that does not rely on the theory of sheaves with transfers:

\begin{thm}
\label{omega.6}
Let $k$ be a perfect field admitting resolution of singularities.
There exist adjoint functors
\[
\omega_\sharp
:
\logSH(k)
\rightleftarrows
\SH(k)
:
\omega^* 
\]
satisfying the following properties:
\begin{enumerate}
\item[\textup{(1)}]
$\omega_\sharp (\Sigma_{\P^1}^\infty Y_+)\simeq \Sigma_{\P^1}^\infty (Y-\partial Y)_+$ for $Y\in \lSm/k$.
\item[\textup{(2)}]
$\omega^*(\Sigma_{\P^1}^\infty (Y-\partial Y)_+)\simeq \Sigma_{\P^1}^\infty Y_+$ for proper $Y\in \lSm/k$.
\item[\textup{(3)}]
$\omega_\sharp$ preserves colimits and is monoidal.
\item[\textup{(4)}]
$\omega^*$ preserves colimits and is monoidal and fully faithful.
\item[\textup{(5)}]
The essential image of $\omega^*$ is the full subcategory spanned by $\A^1$-local objects of $\logSH(k)$.
\end{enumerate}
\end{thm}

Our convention for adjoint functors between categories or $\infty$-categories $F: \cC\rightleftarrows \cD:G$ is that $F$ is a left adjoint of $G$.

The monoidality of $\omega^*$ immediately implies the following results on the log motivic stable homotopy groups:

\begin{cor}
\label{omega.8}
Let $k$ be a perfect field admitting resolution of singularities.
Then for all integers $p$ and $q$,
the log motivic stable homotopy group
\[
\pi_{p,q}^{\log}(k)
:=
\pi_0(\Hom_{\logSH(k)}(\Sigma^{p,q}\unit,\unit))
\]
is isomorphic to the motivic stable homotopy group
\[
\pi_{p,q}(k)
:=
\pi_0(\Hom_{\SH(k)}(\Sigma^{p,q}\unit,\unit)),
\]
where $\unit$ denotes the units in $\SH(k)$ and $\logSH(k)$.
\end{cor}

By \cite[Theorem 6.1.4]{logSH},
the logarithmic $K$-theory $\P^1$-spectrum $\logKGL$ admits a geometric model using Grassmannians.
We provide an alternative streamlined proof of this in Corollary \ref{omega.13} 
as an application of Theorem \ref{omega.6}.

\subsection*{Notation and conventions}

Our standard references for log geometry and $\infty$-categories are Ogus's book \cite{Ogu} and Lurie's books \cite{HTT} and \cite{HA}.
We employ the following notation throughout this paper:

\begin{tabular}{l|l}
$k$ & a perfect field admitting resolution of singularities
\\
$S$ & a finite dimensional noetherian scheme
\\
$\Sm/S$ & the category of separated smooth schemes of finite type over $S$
\\
$\lSch/S$ & the category of separated fs log schemes of finite type over $S$
\\
$\lSm/S$ & the category of separated log smooth fs log schemes of finite
\\
& type over $S$
\\
$\SmlSm/S$ & the full subcategory of $\lSm/S$ spanned by objects $X$ such
\\
& that its underlying scheme is smooth over $S$.
\\
$\Hom_{\cC}(\cF,\cG)$
&
the space of morphisms $\cF\to \cG$ in an $\infty$-category $\cC$
\end{tabular}

\subsection*{Acknowledgements}

This research was conducted in the framework of the DFG-funded research training group GRK 2240: \emph{Algebro-Geometric Methods in Algebra, Arithmetic and Topology}.

\section{A right adjoint of \texorpdfstring{$\omega$}{w}}

For an fs log scheme $Y$,
let $\ul{Y}$ be the underlying scheme of $Y$.
Let $Y-\partial Y$ be the largest open subscheme of $Y$ with the trivial log structure,
which is possible due to \cite[Proposition III.1.2.8(2)]{Ogu}.
If $Y\in \SmlSm/S$,
then we can regard the closed complement $\partial Y$ as a strict normal crossing divisor on $\ul{Y}$.
In this case,
we often use the notation $Y=(\ul{Y},\partial Y)$.

A morphism $f\colon Y_1\to Y_2$ in $\lSm/k$ is an \emph{admissible blow-up} if $f$ is proper (i.e., $\ul{f}$ is proper) and the induced morphism of schemes $Y_1-\partial Y_1\to Y_2-\partial Y_2$ is an isomorphism.
Let $\Adm$ denote the class of admissible blow-ups in $\lSm/k$,
which admits calculus of right fractions in $\lSm/k$ by \cite[Proposition 7.6.6]{logDM}.

We have the functor
\[
\omega
\colon
\lSm/k
\to
\Sm/k
\]
given by $\omega(X):=X-\partial X$ for $X\in \lSm/k$.
Since $\omega$ is invariant under admissible blow-ups,
we have the induced functor
\[
\omega
\colon
\lSm/k[\Adm^{-1}]
\to
\Sm/k.
\]

For $X\in \Sm/k$,
there exists proper $X'\in \Sm/k$ containing $X$ as an open subscheme whose complement is a strict normal crossing divisor by resolution of singularities.
If $Y$ is the fs log scheme over $k$ with the underlying scheme $X'$ and the compactifying log structure \cite[Definition III.1.6.1]{Ogu} associated with the open immersion $X\to X'$,
then $Y$ is proper, $Y\in \SmlSm/k$, and $Y-\partial Y\simeq X$.
We will use this construction frequently.

\begin{prop}
\label{omega.1}
If $Y_1,Y_2\in \lSm/k$ and $Y_2$ is proper,
then there is a natural isomorphism
\[
\alpha\colon
\Hom_{\lSm/k[\Adm^{-1}]}(Y_1,Y_2)
\xrightarrow{\simeq}
\Hom_{\Sm/k}(Y_1-\partial Y_1,Y_2-\partial Y_2)
\]
sending a morphism $f\colon Y_1\to Y_2$ to $\omega(f)$.
\end{prop}
\begin{proof}
Let $g\colon Y_1-\partial Y_1\to Y_2-\partial Y_2$ be a morphism in $\Sm/k$,
and let $u_1\colon Y_1-\partial Y_1\to \ul{Y_1}$ and $u_2\colon Y_2-\partial Y_2\to \ul{Y_2}$ be the obvious open immersions of schemes.
Consider the morphism $(u_1,u_2g)\colon Y_1-\partial Y_1\to \ul{Y_1}\times \ul{Y_2}$,
and let $\ul{Z}$ be the closure of its image.
By resolution of singularities,
there exists $Y\in \SmlSm/k$ with a proper birational morphism of schemes $\ul{Y}\to \ul{Z}$ such that $Y-\partial Y\simeq Y_1-\partial Y_1$.
Since $\ul{Y_2}$ is proper,
$\ul{Z}$ is proper over $\ul{Y_1}$.
Together with \cite[Proposition III.1.6.2]{Ogu},
we have an admissible blow-up $Y\to Y_1$ and a morphism $Y\to Y_2$.
Use these two to form a morphism $f\colon Y_1\to Y_2$ in $\lSm/k[\Adm^{-1}]$ such that $\alpha(f)=g$.
This shows that $\alpha$ is surjective.

Let $f,f'\colon Y_1\to Y_2$ be two morphisms in $\lSm/k[\Adm^{-1}]$.
After replacing $Y_1$ by its admissible blow-up,
we may assume that $f$ and $f'$ are indeed morphisms in $\lSm/k$.
If $f$ and $g$ agree on $Y_1-\partial Y_1$,
then we have $f=g$ since $Y_1-\partial Y_1$ is dense in $Y_1$.
This shows that $\alpha$ is injective.
\end{proof}

Using resolution of singularities and Proposition \ref{omega.1},
we see that the restriction of $w\colon \lSm/k[\Adm^{-1}]\to \Sm/k$ to proper fs log schemes in $\lSm/k[\Adm^{-1}]$ is an equivalence of categories.
By considering its quasi-inverse,
we obtain a functor
\[
c\colon \Sm/k
\to
\lSm/k[\Adm^{-1}]
\]
such that
\[
c(Y-\partial Y)\simeq Y
\]
for proper $Y\in \lSm/k$.
Using Proposition \ref{omega.1} again,
we see that the functors
\[
\omega
:
\lSm/k[\Adm^{-1}]
\rightleftarrows
\Sm/k
:
c
\]
form an adjoint pair.

\section{Compactifications of Nisnevich distinguished squares}
\label{omega.14}

Recall that $S$ is a finite dimensional noetherian scheme.
We often deal with the case when $S=\Spec(k)$.

Let $\sT(S)$ be a stable $\infty$-category.
We refer to \cite[Definition 3.1.4(1)]{logDM} for the notions of strict Nisnevich distinguished squares and strict Nisnevich cd-structure in $\lSm/S$.
The associated topology is the strict Nisnevich topology,
and $sNis$ is its shorthand.
We set $\square:=(\P^1,\infty)$.
We assume that there exists a functor
\[
M\colon \rN(\lSm/S) \to \sT(S)
\]
satisfying the following conditions:
\begin{itemize}
\item (Strict Nisnevich descent)
For every strict Nisnevich distinguished square $Q$ in $\lSm/S$,
the square $M(Q)$ in $\sT(S)$ is cocartesian.
\item (Dividing descent)
For every dividing cover (i.e., a surjective proper log \'etale monomorphism) $f$ in $\lSm/S$,
$M(f)$ is an isomorphism in $\sT(S)$.
\item ($\square$-invariance)
For $Y\in \lSm/S$,
the projection $p\colon Y\times \square\to Y$ induces the isomorphism $M(p)\colon M(Y\times \square)\to M(Y)$.
\end{itemize}

Suppose that we have $Y\in \SmlSm/S$ and $W\in \lSch/S$ such that $\ul{W}$ and $\ul{W}+\partial Y$ are strict normal crossing divisors on $\ul{Y}$.
In this setting,
we use the convenient notation
\[
(Y,W)
:=
(\ul{Y},\ul{W}+\partial X),
\]
which adds the additional boundary $W$ to $Y$.

We will use the notion of an \emph{$n$-cube} in an $\infty$-category $\cC$,
see \cite[\S A.6]{logSH} for the details.
A $0$-cube is just an object of $\cC$.
For every integer $n\geq 0$,
an $(n+1)$-cube $Q$ can be constructed inductively as a morphism of $n$-cubes $Q_0\to Q_1$.
If $\cC$ is stable,
then the \emph{total cofiber $\tcofib(Q)$ of $Q$} has an inductive formula
\[
\tcofib(Q)
\simeq
\cofib(\tcofib(Q_0)\to \tcofib(Q_1)).
\]
The total cofiber of a $0$-cube is its underlying object.
We say that $Q$ is \emph{cocartesian} if $\tcofib(Q)\simeq 0$.

\begin{df}
Let $W\to Y$ be a strict closed immersion in $\SmlSm/S$ such that $\ul{W}$ is strict normal crossing with $\partial Y$.
In other words,
Zariski locally on $Y$,
there exists a strict \'etale morphism $w\colon Y\to \A_\N^m \times \A^{n+l} \times S$ with $\A_\N:=(\A^1,\{0\})$ for some integers $m,n,l\geq 0$ such that $W\simeq w^{-1}(\A_\N^m \times \{0\}^n \times \A^l\times S)$.
The \emph{blow-up of $Y$ along $W$} is
\[
\Bl_Y W
:=
(\Bl_{\ul{Y}}\ul{W},\widetilde{\partial Y}),
\]
where $\widetilde{\partial Y}$ is the sum of the strict transforms of the components of $\partial Y$.
Observe that we have $W\times_Y \Bl_W Y\in \SmlSm/S$,
which can be shown using the local description of $W\to Y$ in terms of $w$.
\end{df}

\begin{prop}
\label{omega.7}
Let $W\to Y$ be a strict closed immersion in $\SmlSm/S$ such that $\ul{W}$ is strict normal crossing with $\partial Y$.
Consider the cartesian square
\[
Q
:=
\begin{tikzcd}
W\times_Y \Bl_W Y\ar[r]\ar[d]&
\Bl_W Y\ar[d,"p"]
\\
W\ar[r]&
Y,
\end{tikzcd}
\]
where $p$ is the projection.
Then the square $M(Q)$ in $\sT(S)$ is cocartesian.
\end{prop}
\begin{proof}
The question is Zariski local on $Y$.
Hence we may assume that there exists a strict \'etale morphism $w\colon Y\to \A_\N^m \times \A^{n+l} \times S$ for some integers $m,n,l\geq 0$ such that $W\simeq w^{-1}(\A_\N^m \times \{0\}^n \times \A^l\times S)$.
We proceed by induction on $m+n$.

The claim is clear if $m+n=0$ since $W\simeq Y$ in this case.
Assume $m+n>0$.
Consider the underlying morphism of schemes $\ul{w}\colon \ul{Y}\to \A^{m+n+l}\times S$,
and we set $\ul{V}:=\ul{w}^{-1}(\{0\}^{m+n} \times  \A^l\times  S)$.
Apply \cite[Construction 7.2.8]{logDM} to the pair 
$(\ul{Y},\ul{V})$ to obtain strict \'etale morphisms $Y\xleftarrow{u} Y_2\xrightarrow{v} Y_1$ in $\SmlSm/S$ with cartesian squares
\[
\begin{tikzcd}
\ul{V}\ar[d]\ar[r,"\id",leftarrow]&
\ul{V}\ar[r,"\id"]\ar[d]&
\ul{V}\ar[d]
\\
\ul{Y}\ar[r,leftarrow,"\ul{u}"]&
\ul{Y_2}\ar[r,"\ul{v}"]&
\ul{Y_1}
\end{tikzcd}
\]
such that $Y_1\simeq \A_\N^m \times \A^n \times \ul{V}$.
We set $W_1:=\A_\N^m \times \{0\}^n \times \ul{V}$.
In this construction,
we have an isomorphism
\[
W_2:=
W\times_Y Y_2
\simeq
W_1\times_{Y_1}Y_2.
\]

Let $C$ be the $3$-cube given by the projection $Q_2:=Q\times_Y Y_2\to Q$,
and let $D$ be the $4$-cube given by the projection $C\times_Y (Y-\ul{V})\to C$.
The induced square
\[
B:=
\begin{tikzcd}
Y_2\times_Y (Y-\ul{V})\ar[d]\ar[r]&
Y_2\ar[d,"u"]
\\
Y-\ul{V}\ar[r]&
Y
\end{tikzcd}
\]
and its pullbacks $B\times_Y W$, $B\times_Y \Bl_W Y$, and $B\times_Y (W\times_Y \Bl_W Y)$ are strict Nisnevich distinguished squares.
It follows that $M(D)$ is cocartesian.

The squares $M(Q\times_Y (Y-\ul{V}))$ and $M(Q\times_Y (Y-\ul{V})\times_Y Y_2)$ are cocartesian by induction.
Hence $M(C\times_Y (Y-\ul{V}))$ is cocartesian.
Together with the fact that $M(D)$ is cocartesian,
we deduce that $M(C)$ is cocartesian.
Hence to show that $M(Q)$ is cocartesian,
it suffices to show that $M(Q_2)$ is cocartesian.
By a similar argument,
we reduce to showing that $M(Q_1)$ is cocartesian,
where
\[
Q_1
:=
\begin{tikzcd}
W_1\times_{Y_1} \Bl_{W_1} {Y_1}\ar[r]\ar[d]&
\Bl_{W_1} Y_1\ar[d]
\\
W_1\ar[r]&
Y_1.
\end{tikzcd}
\]
Replace the pair $(Y,W)$ with $(Y_1,W_1)$ to reduce to the case when $Y\simeq \A^n \times W$.

In this case,
argue as in the second half of the proof of \cite[Theorem 7.3.3]{logDM} to show that $M(Q)$ is cocartesian.
\end{proof}

\begin{const}
The strict Nisnevich descent property implies the Zariski separation property in \cite[\S 7]{logDM}.
Since loc.\ cit.\ works on the level of triangulated categories and not $\infty$-categories,
the Zariski separation property is additionally needed there.

Hence we can use \cite[Theorem 7.6.7]{logDM} to see that the functor $M\colon \rN(\lSm/k)\to \sT(k)$ is invariant under admissible blow-ups.
From this,
we obtain the induced functor
\[
M\colon \rN(\lSm/k[\Adm^{-1}])
\to
\sT(k).
\]
In particular,
we can consider $M(c(X))\in \sT(k)$ for $X\in \Sm/k$.
\end{const}

\begin{prop}
\label{omega.4}
Let $i\colon Z\to X$ be a closed immersion in $\Sm/k$,
Consider the cartesian square
\[
Q
:=
\begin{tikzcd}
Z\times_X \Bl_Z X\ar[d]\ar[r]&
\Bl_Z X\ar[d,"p"]
\\
Z\ar[r,"i"]&
X,
\end{tikzcd}
\]
where $p$ is the projection.
Then the square $M(c(Q))$ in $\sT(k)$ is cocartesian.
\end{prop}
\begin{proof}
By resolution of singularities,
there exists proper $Y\in \SmlSm/k$ such that $Y-\partial Y\simeq X$ and the closure $\ol{Z}$ of $Z$ in $\ul{Y}$ is strict normal crossing with $\partial Y$.
Let $W$ be the fs log scheme $(\ol{Z},\ol{Z}\cap \partial Y)$,
which is an object of $\SmlSm/k$.
Consider the induced cartesian square
\[
R:=
\begin{tikzcd}
W\times_Y \Bl_W Y\ar[d]\ar[r]&
\Bl_W Y\ar[d]
\\
W\ar[r]&
Y,
\end{tikzcd}
\]
Since $c(Q)\simeq R$ as squares in $\lSm/k[\Adm^{-1}]$,
we need to show that $M(R)$ is cocartesian.
This follows from Proposition \ref{omega.7}.
\end{proof}

\begin{prop}
\label{omega.3}
Let
\[
Q
:=
\begin{tikzcd}
(U,U\cap \ul{W})\ar[r]\ar[d]&
U\ar[d,"j"]
\\
(Y,\ul{W})\ar[r]&
Y
\end{tikzcd}
\]
be a cartesian square in $\SmlSm/k$,
where $j$ is an open immersion and $\ul{W}$ is a smooth divisor on $\ul{Y}$ such that $\ul{W}+\partial Y$ is a strict normal crossing divisor.
If $\{U,Y-\ul{W}\}$ is a Zariski cover of $X$,
then the square $M(Q)$ in $\sT(k)$ is cocartesian.
\end{prop}
\begin{proof}
Consider the induced square
\[
Q'
:=
\begin{tikzcd}
U-U\cap \ul{W}\ar[d]\ar[r]&
(U,U\cap \ul{W})\ar[d]
\\
Y-\ul{W}\ar[r]&
(Y,\ul{W}).
\end{tikzcd}
\]
We regard $Q\cup Q'$ as a square.
Since $Q'$ and $Q\cup Q'$ are strict Nisnevich distinguished squares,
$M(Q')$ and $M(Q\cup Q')$ are cocartesian.
It follows that $M(Q)$ is cocartesian too.
\end{proof}

Let $W\to Y$ be a strict closed immersion in $\SmlSm/k$ such that $\ul{W}$ is strict normal crossing with $\partial Y$.
We need to distinguish the two different \emph{deformation spaces}
\[
\rD_Z^{\A^1}(Y)
:=
\Bl_{W\times \{0\}}(Y\times \A^1)-\Bl_{W\times \{0\}}(Y\times \{0\})
\]
and
\[
\rD_Z^{\square}(Y)
:=
\Bl_{W\times \{0\}}(Y\times \square)-\Bl_{W\times \{0\}}(Y\times \{0\}).
\]
The \emph{normal bundle of $W$ in $Y$} is defined to be
\[
\rN_W Y:=
\rN_{\ul{W}}\ul{Y} \times_{\ul{Y}}Y,
\]
where $\rN_{\ul{W}}\ul{Y}$ is the normal bundle of $\ul{W}$ in $\ul{Y}$.

For a morphism $Y_1\to Y_2$ in $\lSm/k[\Adm^{-1}]$,
we set
\[
M(Y_2/Y_1)
:=
\cofib(M(Y_1)\to M(Y_2)).
\]

\begin{prop}
\label{omega.2}
Let $Z\to X$ be a closed immersion in $\Sm/k$.
Then the induced morphisms
\begin{align*}
M(c(X)/c(X-Z))
\to &
M(c(\rD_Z^{\A^1} X)/c(\rD_Z^{\A^1} X-Z\times \A^1))
\\
\leftarrow &
M(c(\rN_Z X)/c(\rN_Z X-Z))
\end{align*}
are isomorphisms in $\sT(k)$.
\end{prop}
\begin{proof}
By resolution of singularities,
there exists proper $Y\in \SmlSm/k$ such that $Y-\partial Y\simeq X$ and the closure $\ol{Z}$ of $Z$ in $\ul{Y}$ is strict normal crossing with $\partial Y$.
Let $W$ be the fs log scheme $(\ol{Z},\ol{Z}\cap \partial Y)$,
which is an object of $\SmlSm/k$.
The induced morphisms
\begin{align*}
M(Y/(\Bl_W Y,E))
\to &
M(\rD_W^\square Y/(\Bl_{W\times \square}(\rD_W^\square Y),E^D))
\\
\leftarrow &
M(\rN_W Y/(\Bl_W(\rN_W Y),E^N))
\end{align*}
are isomorphisms in $\sT(k)$ by \cite[Theorem 7.5.4]{logDM},
where $E$, $E^D$, and $E^N$ are the exceptional divisors.

We have isomorphisms $c(X)\simeq Y$ and $c(X-Z)\simeq (\Bl_W Y,E)$ in $\lSm/k[\Adm^{-1}]$,
and hence we have an isomorphism
\[
M(c(X)/c(X-Z))
\simeq
M(Y/(\Bl_W Y,E))
\]
in $\sT(k)$.
Observe that $W\times \square \simeq \Bl_{W\times \{0\}}(W\times \square)$ is a strict closed subscheme of $\Bl_{W\times \{0\}}(Y\times \square)$.
We have isomorphisms
\begin{gather*}
c(\rD_Z^{\A^1} X)\simeq (\Bl_{W\times \{0\}}(Y\times \square),\Bl_{W\times \{0\}} (Y\times \{0\})),
\\
c(\rD_Z^{\A^1} X-Z\times \A^1)\simeq (\Bl_{W\times \square}(\Bl_{W\times \{0\}}(Y\times \square)),E^D+\Bl_{W\times \{0\}} (Y\times \{0\}))
\end{gather*}
in $\lSm/k[\Adm^{-1}]$.
Since $W\times \square$ and $\Bl_{W\times \{0\}}(Y\times \{0\})$ are disjoint in $\Bl_{W\times \{0\}}(Y\times \square)$,
$E^D$ and $\Bl_{W\times \{0\}} (Y\times \{0\})$ are disjoint in $\Bl_{W\times \square}(\Bl_{W\times \{0\}}(Y\times \square))$.
Proposition \ref{omega.3} implies that we have an isomorphism
\[
M(c(\rD_Z^{\A^1} X)/c(\rD_Z^{\A^1} X-Z\times \A^1))
\simeq
M(\rD_W^\square Y/(\Bl_{W\times \square}(\rD_W^\square Y),E^D)).
\]
We have isomorphisms
\begin{gather*}
c(\rN_Z X)
\simeq
(\P(\rN_W Y \oplus \cO),H),
\\
c(\rN_Z X-Z)
\simeq
(\Bl_W \P(\rN_W Y \oplus \cO),E^N+H)
\end{gather*}
in $\lSm/k[\Adm^{-1}]$,
where $H$ is the hyperplane at $\infty$.
Since $E^N$ and $H$ are disjoint in $\Bl_W \P(\rN_W Y \oplus \cO)$,
Proposition \ref{omega.3} implies that we have an isomorphism
\[
M(c(\rN_Z X)/M(c(\rN_Z X-Z)))
\simeq
M(\rN_W Y/(\Bl_W(\rN_W Y),E^N)).
\]
Combine what we have discussed to conclude.
\end{proof}

\begin{prop}
\label{omega.5}
Let
\[
Q
:=
\begin{tikzcd}
U'\ar[d]\ar[r]&
X'\ar[d,"f"]
\\
U\ar[r]&
X
\end{tikzcd}
\]
be a Nisnevich distinguish square in $\Sm/k$.
Then the square $M(c(Q))$ in $\sT(k)$ is cocartesian.
\end{prop}
\begin{proof}
By resolution of singularities,
there exist morphisms $X_n\to X_{n-1}\to \cdots \to X_0 := X$ in $\Sm/k$ satisfying the following conditions:
\begin{enumerate}
\item[(i)] $X_{i+1}$ is the blow-up along a smooth center $Z_i$ in $X_i$ for $0\leq i<n$.
\item[(ii)] The induced morphism $U\times_X X_{i+1}\to U\times_X X_i$ is an isomorphism for $0\leq i<n$.
\item[(iii)] The complement of $U$ in $X_n$ is a strict normal crossing divisor.
\end{enumerate}
Consider the induced cartesian square
\[
Q_i
:=
\begin{tikzcd}
Z_i\times_{X_i}X_{i+1}\ar[d]\ar[r]&
X_{i+1}\ar[d]
\\
Z_i\ar[r]&
X_i.
\end{tikzcd}
\]
Since the induced morphism of schemes $f^{-1}(X-U)\to X-U$ is an isomorphism with the reduced scheme structures,
we have an isomorphism $Z_i\times_{X} X'\simeq Z_i$.
Apply Proposition \ref{omega.4} to $Q_i$ and $Q_i\times_X X'$ to obtain an isomorphism
\[
M(c(X_i)/c(X_{i+1}))
\simeq
M(c(X'\times_X X_i)/c(X'\times_X X_{i+1})).
\]
From this,
we obtain isomorphisms
\[
M(c(X)/c(X'))
\simeq
M(c(X_1)/c(X'\times_X X_1))
\simeq
\cdots
\simeq
M(c(X_n)/c(X'\times_X X_n)).
\]
Hence to show that $M(c(Q))$ is cocartesian,
it suffices to show that $M(c(Q\times_X X_n))$ is cocartesian.
Replace $X$ with $X_n$ to reduce to the case when the complement of $U$ in $X$ is a strict normal crossing divisor.

By resolution of singularities,
there exists proper $Y\in \SmlSm/k$ such that $Y-\partial Y\simeq X$ and $Y-U$ is a strict normal crossing divisor $\partial Y + W_1+\cdots +W_m$.
We set $U_i:=X-(W_1\cup \cdots \cup W_i)\cap X$ and $Y_i:=(\ul{Y},\partial Y+W_1+\cdots+W_i)$ for integers $0\leq i\leq n$.
Observe that we have an isomorphism $Y_i\simeq c(U_i)$ in $\lSm/k[\Adm^{-1}]$.
The induced square
\[
\begin{tikzcd}
U_{i+1}\times_X X'\ar[r]\ar[d]&
U_i\times_X X'\ar[d]
\\
U_{i+1}\ar[r]&
U_i
\end{tikzcd}
\]
is a Nisnevich distinguished square in $\Sm/k$ for $0\leq i<n$.
Replace $U\to X$ with $U_{i+1}\to U_i$ to reduce to the case when the complement of $U$ in $X$ is a smooth divisor $Z$.

We set $Z':=Z\times_X X'$,
which is isomorphic to $Z$.
We have an induced commutative square
\[
\begin{tikzcd}
M(c(X)/c(U))\ar[d]\ar[r,"\simeq"]&
M(c(\rN_Z X)/c(\rN_Z X-Z))\ar[d]
\\
M(c(X')/c(U'))\ar[r,"\simeq"]&
M(c(\rN_{Z'} X')/c(\rN_{Z'} X'-Z'))
\end{tikzcd}
\]
whose horizontal arrows are isomorphisms by Proposition \ref{omega.2}.
The right vertical arrow is an isomorphism since $\rN_Z X\simeq \rN_{Z'} X'$.
Hence the left vertical arrow is an isomorphism,
i.e.,
$M(Q)$ is cocartesian.
\end{proof}

\section{Proof of Theorem \ref{omega.6}}

For an $\infty$-category $\cC$ with a class of morphisms $A$ in $\cC$, let $A^{-1}\cC$ be the full subcategory of $\cC$ consisting of $\cA$-local objects \cite[\S 5.5.4]{HTT}.
We have the localization functor $L_A\colon \cC\to A^{-1}\cC$,
whose right adjoint $\iota_A\colon A^{-1}\cC\to \cC$ is the inclusion functor.

\begin{lem}
\label{omega.10}
Let $F: \cC\rightleftarrows \cD: G$ be an adjunction between presentable $\infty$-categories such that $G$ preserves colimits too,
and let $A$ and $B$ be classes of morphisms in $\cC$ and $\cD$.
If $L_B F$ maps every morphism in $A$ to an isomorphism and $L_A G$ maps every morphism in $B$ to an isomorphism,
then there is an induced adjoint colimit preserving functors $F: A^{-1}\cC \rightleftarrows B^{-1}\cD: G$ such that the two squares in the diagram
\[
\begin{tikzcd}
\cC\ar[d,"L_A"']\ar[r,"F",shift left=0.5ex]\ar[r,leftarrow,"G"',shift right=0.5ex]&
\cD\ar[d,"L_B"]
\\
A^{-1}\cC
\ar[r,"F",shift left=0.5ex]\ar[r,leftarrow,"G"',shift right=0.5ex]&
B^{-1}\cD
\end{tikzcd}
\]
commutes.
\end{lem}
\begin{proof}
The assumption on $L_B F$ implies that there exists a unique colimit preserving functor $F\colon A^{-1}\cC\to B^{-1}\cD$ such that $L_BF\simeq FL_A$ by \cite[Proposition 5.5.4.20]{HTT}.
Let $G'\colon B^{-1}\cD\to A^{-1}\cC$ be its right adjoint.
Similarly,
the assumption on $L_A G$ implies that there exists a unique colimit preserving functor $G\colon B^{-1}\cD\to A^{-1}\cC$ such that $L_AG\simeq GL_B$.
Since $L_B\iota_B\simeq \id$,
we obtain $L_A G\iota_B \simeq G$.
From  $L_BF\simeq FL_A$,
we obtain $G\iota_B\simeq \iota_A G'$.
Since $L_A\iota_A\simeq \id$,
we obtain $L_A G\iota_B\simeq G'$.
Hence we obtain $G\simeq G'$.
\end{proof}

\begin{lem}
\label{omega.12}
Let $\cC$ be a category,
and let $W$ be a class of morphisms in $\cC$.
If $W$ admits calculus of left or right fractions in $\cC$ and $W$ satisfies the two-out-of-three property in $\cC$,
then there is a natural equivalence of $\infty$-categories
\[
\rN(\cC[W^{-1}])
\simeq
\rN(\cC)[W^{-1}],
\]
where $\cC[W^{-1}]$ in the left-hand side is the $1$-categorical localization.
\end{lem}
\begin{proof}
Dwyer-Kan \cite[\S 7]{MR0578563} show that the hammock localization $L^H(\cC,W)$ is weak equivalent to $\cC[W^{-1}]$.
We regard the $1$-categories $\cC$ and $W$ as fibrant simplicial categories (i.e., the hom spaces are Kan complexes).
We obtain the desired natural equivalence using \cite[Remark 1.3.4.2]{HA} and Hinich's comparison \cite[Proposition 1.2.1]{MR3460765} between the $\infty$-categorical localization and the hammock localization.
\end{proof}

Let $\Sp$ denote the $\infty$-category of spectra.
For a category $\cC$,
let $\PSh(\cC,\Sp)$ denote the $\infty$-category of presheaves of spectra on $\cC$.
For $X\in \cC$,
let $\Sigma_{S^1}^\infty X_+\in \PSh(\cC,\Sp)$ be the presheaf represented by $X$.
If $\cC$ has a topology,
then let $\Sh_t(\cC,\Sp)$ denote the $\infty$-category of $t$-sheaves of spectra on $\cC$.

\begin{lem}
\label{omega.9}
There is an equivalence of $\infty$-categories
\[
\PSh(\lSm/k[\Adm^{-1}],\Sp)
\simeq
\Adm^{-1}\PSh(\lSm/k,\Sp).
\]
\end{lem}
\begin{proof}
The left-hand side is $\Fun(\rN (\lSm/k[\Adm^{-1}])^{op},\Sp)$.
Observe that $\Adm$ admits calculus of right fractions in $\cC$ by \cite[Proposition 7.6.6]{logDM} and satisfies the two-out-of-three property in $\cC$.
By Lemma \ref{omega.12} and the definition of the $\infty$-categorical localization \cite[Definition 1.3.4.1]{HA},
$\Fun(\rN(\lSm/k[\Adm^{-1}])^{op},\Sp)$ is equivalent to the full subcategory of $\Fun(\rN(\lSm/k)^{op},\Sp)\simeq \PSh(\lSm/k,\Sp)$ consisting of functors $F\colon \rN(\lSm/k)^{op}\to \Sp$ such that $F(u)$ is an isomorphism of spectra for every admissible blow-up $u$ in $\lSm/k$.
From this description,
we obtain the desired equivalence of $\infty$-categories.
\end{proof}

We have the localization functors
\begin{gather*}
L_{Nis,\A^1}
\colon
\PSh(\Sm/k,\Sp)
\to
(\A^1)^{-1}\Sh_{Nis}(\Sm/k,\Sp)
=:
\SH_{S^1}(k),
\\
L_{sNis,\square}
\colon
\Adm^{-1}
\PSh(\lSm/k,\Sp)
\to
(\Adm\cup\square)^{-1}\Sh_{sNis}(\lSm/k,\Sp)
=:
\logSH_{S^1}(k),
\end{gather*}
where $\A^1$ (resp.\ $\square$) denotes the class of projections $X\times \A^1\to X$ (resp.\ $X\times \square\to X$) in $\Sm/k$ (resp.\ $\lSm/k$).

\begin{thm}
\label{omega.11}
Let $k$ be a perfect field admitting resolution of singularities.
There exist adjoint functors
\[
\omega_\sharp
:
\logSH_{S^1}(k)
\rightleftarrows
\SH_{S^1}(k)
:
\omega^* 
\]
satisfying the following properties:
\begin{enumerate}
\item[\textup{(1)}]
$\omega_\sharp (\Sigma_{S^1}^\infty Y_+)\simeq \Sigma_{S^1}^\infty (Y-\partial Y)_+$ for $Y\in \lSm/k$.
\item[\textup{(2)}]
$\omega^*(\Sigma_{S^1}^\infty (Y-\partial Y)_+)\simeq \Sigma_{S^1}^\infty Y_+$ for proper $Y\in \lSm/k$,
\item[\textup{(3)}]
$\omega_\sharp$ preserves colimits and is monoidal.
\item[\textup{(4)}]
$\omega^*$ preserves colimits and is monoidal and fully faithful.
\item[\textup{(5)}]
The essential image of $\omega^*$ is the full subcategory spanned by $\A^1$-local objects of $\logSH_{S^1}(k)$.
\end{enumerate}
\end{thm}
\begin{proof}
We begin with the adjoint pair
\[
\omega
:
\lSm/k[\Adm^{-1}]
\rightleftarrows
:
\Sm/k
:
c.
\]
This induces the adjoint pair
\[
\omega_\sharp
:
\PSh(\lSm/k[\Adm^{-1}],\Sp)
\rightleftarrows
\PSh(\Sm/k,\Sp)
:
c_\sharp
\]
such that $\omega_\sharp (\Sigma_{S^1}^\infty Y_+)\simeq \Sigma_{S^1}^\infty (Y-\partial Y)_+$ for $Y\in \lSm/k$ and $c_\sharp(\Sigma_{S^1}^\infty (Y-\partial Y)_+)\simeq \Sigma_{S^1}^\infty Y_+$ for proper $Y\in \lSm/k$.
Together with Lemma \ref{omega.9},
we obtain the adjoint pair
\[
\omega_\sharp
:
\Adm^{-1}\PSh(\lSm/k,\Sp)
\rightleftarrows
\PSh(\Sm/k,\Sp)
:
c_\sharp.
\]

Observe that the functor
\[
L_{sNis,\square}\Sigma_{S^1}^\infty (-)
\colon
\rN(\lSm/k)
\to
\logSH_{S^1}(k)
\]
satisfies the conditions for $M$ in \S \ref{omega.14}.
Hence we can use Proposition \ref{omega.5} to deduce that
$L_{sNis,\square}c_\sharp\Sigma^\infty_{S^1}(-)_+$ sends every Nisnevich distinguished square to a cocartesian square.
Since $c$ sends $\A^1$ to $\square$ and $\omega$ sends dividing Nisnevich squares to Nisnevich squares and $\square$ to $\A^1$,
Lemma \ref{omega.10} yields the diagram
\[
\begin{tikzcd}
\Adm^{-1}\PSh(\lSm/k,\Sp)\ar[d,"L_{sNis,\square}"']\ar[r,"\omega_\sharp",shift left=0.5ex]\ar[r,leftarrow,"c_\sharp"',shift right=0.5ex]&
\PSh(\Sm/k,\Sp)\ar[d,"L_{Nis,\A^1}"]
\\
\logSH_{S^1}(k)
\ar[r,"\omega_\sharp",shift left=0.5ex]\ar[r,leftarrow,"\omega^*"',shift right=0.5ex]&
\SH_{S^1}(k)
\end{tikzcd}
\]
such that the two squares in this diagram commute and the functors $\omega_\sharp$ and $\omega^*$ in the lower horizontal row preserve colimits.
These two functors satisfy (1) and (2).

For $Y_1,Y_2\in \lSm/k$,
there is a natural isomorphism
\[
(Y_1-\partial Y_1)\times (Y_2-\partial Y_2)
\simeq
(Y_1\times Y_2)-\partial (Y_1\times Y_2).
\]
This implies that the functor $c\colon \Sm/k\to \lSm/k[\Adm^{-1}]$ preserves products.
Hence the functor $c_\sharp \colon \PSh(\Sm/k,\Sp)\to \Adm^{-1}\PSh(\lSm/k,\Sp)$ is monoidal.
Together with \cite[Example 1.3.4.3, Proposition 4.1.7.4]{HA},
we deduce that $\omega^*\colon \SH_{S^1}(k)\to \logSH_{S^1}(k)$ is monoidal.
We can similarly show that $\omega_\sharp \colon \logSH_{S^1}(k)\to \SH_{S^1}(k)$ is monoidal.
Since $\omega(c(Y))\simeq Y$ for $Y\in \lSm/k$,
we have $\omega_\sharp \omega^*\simeq \id$.
Hence $\omega^*$ is fully faithful,
so we have (3) and (4).

Since the functor $\omega\colon \lSm/k\to \Sm/k$ sends $\A^1$ to $\A^1$,
Lemma \ref{omega.10} yields the diagram
\[
\begin{tikzcd}
\logSH_{S^1}(k)
\ar[d,"L_{\A^1}"']\ar[r,"\omega_\sharp",shift left=0.5ex]\ar[r,leftarrow,"\omega^*"',shift right=0.5ex]&
\SH_{S^1}(k)
\\
(\A^1)^{-1}\logSH_{S^1}(k)
\ar[ru,"\omega_\sharp",shift left=0.5ex,bend right]\ar[ru,leftarrow,"\omega^*"',shift right=0.5ex,bend right]
\end{tikzcd}
\]
such that the two triangles commute.
We have $\iota_{\A^1}\omega^*\simeq \omega^*$ since $\omega_\sharp \simeq \omega_\sharp L_{\A^1}$,
where
\[
\iota_{\A^1}\colon (\A^1)^{-1}\logSH(k)\to \logSH(k)
\]
is the inclusion functor.
By \cite[Proposition 2.5.7]{logA1},
the functor $\omega^*\colon \SH_{S^1}(k)\to (\A^1)^{-1}\logSH_{S^1}(k)$ is an equivalence of $\infty$-categories.
This implies (5).
\end{proof}

\begin{proof}[Proof of Theorem \ref{omega.6}]
We can define $\SH(k)$ and $\logSH(k)$ by $\otimes$-inverting $\Sigma_{S^1}^\infty \P^1/1$ in  $\SH_{S^1}(k)$ and $\logSH_{S^1}(k)$ using \cite[Definition 2.6]{zbMATH06374152},
where
\[
\Sigma_{S^1}^\infty \P^1/1:=\cofib(\Sigma_{S^1}^\infty \{1\}_+\to \Sigma_{S^1}^\infty (\P^1)_+).
\]
For symmetric monoidal $\infty$-categories $\cC$ and $\cD$,
let $\Fun^{\rL,\otimes}(\cC,\cD)$ denote the $\infty$-category of colimit preserving symmetric monoidal functors.
By \cite[Proposition 2.9]{zbMATH06374152},
the induced functor
\[
\Fun^{\rL,\otimes}(\SH(k),\logSH(k))
\to
\Fun^{\rL,\otimes}(\SH_{S^1}(k),\logSH(k))
\]
is fully faithful and its essential image is spanned by the functors $\SH_{S^1}(k)\to \logSH(k)$ sending $\Sigma_{S^1}^\infty \P^1/1$ to a $\otimes$-invertible object of $\logSH(k)$.
Likewise,
the induced functor
\[
\Fun^{\rL,\otimes}(\logSH(k),\SH(k))
\to
\Fun^{\rL,\otimes}(\logSH_{S^1}(k),\SH(k))
\]
is fully faithful and its essential image is spanned by the functors $\logSH(k)\to \SH_{S^1}(k)$ sending $\Sigma_{S^1}^\infty \P^1/1$ to a $\otimes$-invertible object of $\SH(k)$.
With this universality in hand,
we obtain the adjoint functors
\[
\omega_\sharp
:
\logSH(k)
\rightleftarrows
\SH(k)
:
\omega^*.
\]
satisfying (1)--(4) since $\omega\colon \lSm/k\to \Sm/k$ and $c\colon \Sm/k\to \lSm/k[\Adm^{-1}]$ send $\P^1$ to $\P^1$.

As in the proof of Theorem \ref{omega.11},
we have the induced diagram
\[
\begin{tikzcd}
\logSH(k)
\ar[d,"L_{\A^1}"']\ar[r,"\omega_\sharp",shift left=0.5ex]\ar[r,leftarrow,"\omega^*"',shift right=0.5ex]&
\SH(k)
\\
(\A^1)^{-1}\logSH(k)
\ar[ru,"\omega_\sharp",shift left=0.5ex,bend right]\ar[ru,leftarrow,"\omega^*"',shift right=0.5ex,bend right]
\end{tikzcd}
\]
such that the two triangles commute.
Since \cite[Proposition 2.5.7]{logA1} also shows that the functor $\omega^*\colon \SH_{S^1}(k)\to (\A^1)^{-1}\logSH_{S^1}(k)$ is an equivalence of $\infty$-categories,
we have (5).
\end{proof}

Consider the \emph{logarithmic $K$-theory $\P^1$-spectrum}
\[
\logKGL:=\omega^*\KGL\in \logSH(k),
\]
where $\KGL$ is the $K$-theory $\P^1$-spectrum in $\SH(k)$ \cite[\S 6.2]{zbMATH01194164}.
The following is a direct consequence of \cite[Theorem 6.1.4]{logSH}.
Here, we provide an alternative streamlined proof:

\begin{cor}
\label{omega.13}
Let $k$ be a perfect field admitting resolution of singularities.
Then there is a natural isomorphism
\[
\logKGL
\simeq
\colim (\Sigma^\infty_{\P^1}(\Z\times \Gr) \to \Sigma_{\P^1}^{-1}\Sigma^\infty_{\P^1}(\Z\times \Gr) \to \Sigma_{\P^1}^{-2}\Sigma^\infty_{\P^1}(\Z\times \Gr) \to \cdots)
\]
in $\logSH(k)$,
where $\Gr$ is the infinite Grassmannian with trivial log structure.
\end{cor}
\begin{proof}
Since $\Gr$ is a colimit of proper smooth schemes,
we have a natural isomorphism $\omega^* \Sigma_{\P^1}^\infty (\Z\times \Gr)\simeq \Sigma_{\P^1}^\infty (\Z\times \Gr)$ in $\logSH(k)$ by Theorem \ref{omega.6}(2),(4).
Using the description of $\KGL:=\mathbf{BGL}$ in \cite[\S 6.2]{zbMATH01194164},
we have an isomorphism
\[
\KGL
\simeq
\colim (\Sigma^\infty_{\P^1}(\Z\times \Gr) \to \Sigma_{\P^1}^{-1}\Sigma^\infty_{\P^1}(\Z\times \Gr) \to \Sigma_{\P^1}^{-2}\Sigma^\infty_{\P^1}(\Z\times \Gr) \to \cdots)
\]
in $\SH(k)$.
Apply $\omega^*$ to this isomorphism,
and use Theorem \ref{omega.6}(4) to conclude.
\end{proof}

\bibliography{bib}

\begin{thebibliography}{10}

\bibitem{logSH}
{\sc F.~Binda, D.~Park, and P.~A. {\O}stv{\ae}r}, {\em Logarithmic motivic
  homotopy theory}.
\newblock arxiv:2303.02729.

\bibitem{logDM}
\leavevmode\vrule height 2pt depth -1.6pt width 23pt, {\em Triangulated
  categories of logarithmic motives over a field}, Ast{\'e}risque, 433 (2022).

\bibitem{MR0578563}
{\sc W.~G. Dwyer and D.~M. Kan}, {\em Calculating simplicial localizations}, J.
  Pure Appl. Algebra, 18 (1980), pp.~17--35.

\bibitem{MR3460765}
{\sc V.~Hinich}, {\em Dwyer-{K}an localization revisited}, Homology Homotopy
  Appl., 18 (2016), pp.~27--48.

\bibitem{HTT}
{\sc J.~{Lurie}}, {\em {Higher topos theory}}, vol.~170, Princeton, NJ:
  Princeton University Press, 2009.

\bibitem{HA}
\leavevmode\vrule height 2pt depth -1.6pt width 23pt, {\em {Higher algebra}},
  2017.
\newblock http://www.math.harvard.edu/$\sim$lurie/.

\bibitem{MVW}
{\sc C.~Mazza, V.~Voevodsky, and C.~Weibel}, {\em Lecture notes on motivic
  cohomology}, vol.~2 of Clay Mathematics Monographs, American Mathematical
  Society, Providence, RI; Clay Mathematics Institute, Cambridge, MA, 2006.

\bibitem{MV}
{\sc F.~Morel and V.~Voevodsky}, {\em {${\bf A}^1$}-homotopy theory of
  schemes}, Inst. Hautes \'{E}tudes Sci. Publ. Math.,  (1999), pp.~45--143
  (2001).

\bibitem{Ogu}
{\sc A.~Ogus}, {\em Lectures on Logarithmic Algebraic Geometry}, Cambridge
  Studies in Advanced Mathematics, Cambridge University Press, 2018.

\bibitem{logA1}
{\sc D.~Park}, {\em $\mathbb{A}^1$-homotopy theory of log schemes}, 2022.
\newblock arXiv:2205.14750.

\bibitem{zbMATH06374152}
{\sc M.~{Robalo}}, {\em {\(K\)-theory and the bridge from motives to
  noncommutative motives}}, {Adv. Math.}, 269 (2015), pp.~399--550.

\bibitem{zbMATH01194164}
{\sc V.~{Voevodsky}}, {\em {\(\mathbb{A}^1\)-homotopy theory}}, {Doc. Math.},
  Extra Vol. (1998), pp.~579--604.

\end{thebibliography}
\bibliographystyle{siam}

\end{document}